\documentclass[12pt,oneside]{amsart}
\usepackage[english]{babel}
\usepackage{cite}
\usepackage{latexsym}
\usepackage{amssymb,amsthm,amsmath}
\usepackage{hyperref}
\usepackage{graphicx}
\usepackage{longtable}
\usepackage{xcolor}
\usepackage{enumerate}
\usepackage{mathdots}

\newtheorem{theorem}{Theorem}[section]

\newtheorem{corollary}[theorem]{Corollary}

\newtheorem{proposition}[theorem]{Proposition}

{\theoremstyle{definition}
\newtheorem*{definition*}{Definition}
}

\newtheorem*{corollary*}{Corollary}
\newtheorem*{lemma*}{Lemma}
\newtheorem*{remark*}{Remark}

\definecolor{dgreen}{rgb}{0.13,0.7,.63}

\def\daniele#1 {\fbox {\footnote {\ }}\ \footnotetext { From Daniele: {\color{blue}#1}}}
\def\steven#1 {\fbox {\footnote {\ }}\ \footnotetext { From G\"{o}khan: {\color{red}#1}}}

\title{The Diophantine Equation $(x+1)^k+(x+2)^k+\cdots+(\ell x)^k=y^n$ Revisited}

\author{Daniele Bartoli}
\address{Department of Mathematics and Informatics, University of Perugia, Perugia,  Italy}
\email {daniele.bartoli@unipg.it}

\author{G\"{o}khan Soydan}
\address{Department of Mathematics, Bursa Uluda\u{g} University, 16059, Bursa, TURKEY }
\email {gsoydan@uludag.edu.tr}
\date{\today}

\begin{document}

\maketitle

\begin{abstract}
Let $k,\ell\geq2$ be fixed integers and $C$ be an effectively computable constant depending only on $k$ and $\ell$. In this paper, we prove that all solutions of the equation $(x+1)^{k}+(x+2)^{k}+...+(\ell x)^{k}=y^{n}$ in integers $x,y,n$ with $x,y\geq1, n\geq2, k\neq3$ and $\ell\equiv 1 \pmod 2$ satisfy $\max\{x,y,n\}<C$. The case when $\ell$ is even has already been completed by Soydan (Publ. Math. Debrecen 91 (2017), pp. 369-382).
\end{abstract}
\section{Introduction}
Let $\mathbb{Z}$ and $\mathbb{N}$ be the sets of integers and positive integers, respectively. Many authors studied the Diophantine equation
\begin{equation*}
	1^{k}+2^{k}+...+x^{k}=y^{n}, \quad x,y\in\mathbb{Z},\quad k, n\geq 2
\end{equation*}
(see, e.g., \cite{BGP}, \cite{BHMP}, \cite{Br}, \cite{GTV}, \cite{H}, \cite{JPW}, \cite{Lu}, \cite{Pi1}, \cite{Pi2}, \cite{Sch}, \cite{Ur}, \cite{VGT}). 

A more general case is to consider the Diophantine equation
\begin{equation}  \label{eq.1.1}
	(x+1)^k+(x+2)^k+ ... + (x+r)^k=y^{n} \quad x,y\in\mathbb{Z},\quad k, n\geq 2.
\end{equation}
In 2013, Zhang and Bai \cite{ZB} solved equation \eqref{eq.1.1} with $k=2$ and $r=x$. In 2014, the equation
\begin{equation}\label{eq.1.2}
	(x-1)^k+x^k+(x+1)^k=y^{n}  \quad x, y, n \in\mathbb{Z},  \quad n \geq 2,
\end{equation}
was solved completely by Zhang \cite{Zh1} for $k=2, 3, 4$ (Actually, firstly, Cassels considered equation \eqref{eq.1.2} in 1985, and he proved that $x=0,1,2,24$ are the only integer solutions to this equation for $k=3$ and $n=2$). In the next year, Bennett, Patel and Siksek \cite{BPS}  extended Zhang's result, completely solving equation \eqref{eq.1.2} in the cases $k=5$ and $k= 6$. In 2016, Bennett, Patel and Siksek \cite{BPS2} considered equation \eqref{eq.1.1}. They gave the integral solutions to equation \eqref{eq.1.1} using linear forms in logarithms, sieving, and Frey curves when $k=3$, $2\leq r \leq 50,$ $x\geq1$, and $n$ is prime.\

Let $k\geq2$ be even and $r$ be a non-zero integer. In 2017, Patel and Siksek \cite{PS} showed that
for almost all $d\geq2$ (in the sense of natural density), the equation
\begin{equation*}
	x^k +(x+r)^k +...+(x+(d-1)r)^k =y^{n}, \quad x, y, n \in\mathbb{Z},  \quad n \geq 2,
\end{equation*}
has no solutions. Recently, a generalization of equation \eqref{eq.1.2} was also considered by some authors. Zhang \cite{Zh2}, Koutsianas and Patel (see \cite{Kou}, \cite{KV}) studied the integer solutions of the equation
\begin{equation}\label{eq.1.4}
	(x-d)^{k}+x^{k}+(x+d)^{k}=y^n,\quad x,y\in \mathbb{Z},\quad n\geq2,
\end{equation}
for the cases $k=4$ and $k=2$.
Zhang gave some results on equation \eqref{eq.1.4} with $k=4$ by using modular approach. Koutsianas also used modular approach and he proved that equation \eqref{eq.1.4} has no solutions when $d=p^{b}$, $k=2$, and $n\geq 7$  for $p\in\{7,11,13,17,19,23,31,37,41,43,47\}$ and $b\geq 0$ unless $n$ is in Table 1 in his paper \cite{Kou}.

More recently, Garcia and Patel \cite{GP} complemented the work of Cassels, Koutsianas and Zhang by considering the case when $k=3$ and showing that equation \eqref{eq.1.4} with $n\ge5$ a prime and $0<d \le 10^6$ has only trivial solutions $(x,y,n)$ which satisfy $xy=0$. Then, Kundu and Patel \cite{KP} determined all primitive solutions to the equation $(x+r)^2+(x+2r)^2+\cdots+(x+dr)^2=y^{n}$ for $ 2\le d\le 10$ and for $1\le r\le 10^4$ (except the case $d=6$, where they considered only $1 \le r \le 5000$). 

Let $k,\ell$ be fixed integers such that $k\ge 1$, $\ell>0$ and $\ell$ even. In 2017, Soydan \cite{S} considered the equation
\begin{equation}\label{eq.1.3}
(x+1)^k+(x+2)^k+...+(\ell x)^k =y^{n}, \quad x, y,n\in\mathbb{Z}.
\end{equation}
He proved that it has only finitely many solutions where $x,y\ge 1$, $k\neq 1,3$ and $n\ge 2$. He also showed that equation \eqref{eq.1.3} has infinitely many solutions with $k=1,3$ and $n\ge 2$. In the next year, B\'{e}rczes, Pink, Savas and Soydan \cite{BPSS} considered equation \eqref{eq.1.3} with $\ell=2$. They proved that it has no solutions if $2\leq x\leq 13$, $y\geq 2$, and $n\geq 3$.

In this work, we reconsider the Diophantine equation
\begin{equation}\label{xx}
(x+1)^{k}+(x+2)^{k}+...+(\ell x)^{k}=y^n
\end{equation}
in integers $x, y\ge 1$ and $n\ge 2$.

\section{Preliminaries}
We shall use the following important results of Schinzel-Tijdeman \cite{ST} and Brindza \cite{Br2} on the superelliptic equation
\begin{equation}\label{1.1}
f(x)=y^{n}.
\end{equation}

\begin{theorem}[Schinzel-Tijdeman, \cite{ST}]\label{theo.1}
Let $f(x)\in\mathbb{Q}[x]$ be a polynomial having at least 2 distinct roots. Then there exists an effective constant $N(f)$ such that any solution of \eqref{1.1} in $x,n\in\mathbb{Z}$, $y\in\mathbb{Q}$ satisfies $n\leq N(f)$.
\end{theorem} 

\begin{corollary}[Schinzel-Tijdeman, \cite{ST}]\label{cor.1}
Let $f(x)\in\mathbb{Q}[x]$ be a polynomial having at least 3 simple roots. Then \eqref{1.1} has at most finitely many solutions in $x,n\in\mathbb{Z}$, $y\in\mathbb{Q}$ satisfying $n > 1$. If $f(x)$ has 2 simple roots then \eqref{1.1} has only finitely many solutions with $n>2$. In both cases the solutions can be explicitly determined.
\end{corollary}

\begin{theorem}[Brindza, \cite{Br2}]\label{theo.2}
	Let $H(x)\in\mathbb{Q}[x]$,
	\begin{align*}
	H(x)=a_{0}x^{N}+...+a_{N}=a_{0}\prod_{i=1}^m (x-\alpha_{i})^{r_{i}},
	\end{align*}
	with $a_{0}\neq0$ and $\alpha_{i}\neq\alpha_{j}$ for $i\neq j$. Let $0\neq b\in\mathbb{Z}$, $2\leq n\in\mathbb{Z}$ and define $t_{i}=\frac{n}{(n,r_{i})}$. Suppose that $\{t_{1},...,t_{m}\}$ is not a permutation of the m-tuples $(a)$ $\{t,1,...,1\}$, $t\geq1$; $(b)$ $\{2,2,1,...,1\}.$
	
	Then all solutions $(x,y)\in\mathbb{Z}^{2}$ of the equation
	\begin{align*}
	H(x)=by^{n}
	\end{align*}
	satisfy  $\max\{|x|,|y|\}<C$, where $C$ is an effectively computable constant depending only on $H$, $b$, and $n$.
\end{theorem}

\section{Main results}
Consider $H(x)=(x+1)^{k}+(x+2)^{k}+\cdots +(\ell x)^{k}$. By \cite[Formula 2.3]{Rd}, $H(x)=B_{k+1}(\ell x+1)-B_{k+1}(x+1)$, where 
\begin{align*}
	B_{q}(x)=x^{q}-\frac{1}{2}qx^{q-1}+\dfrac{1}{6}\binom {q} {2}x^{q-2}+\cdots =\sum\limits_{i=0}^q\binom {q} {i}x^{q-i}B_{i}
	\end{align*}
	is the $q$-th Bernoulli polynomial with $q=k+1$; see \cite[formula 2.71]{Rd}.
Therefore,
$$H(x)=\sum _{i=0}^{k+1}\binom{k+1}{i} (\ell x+1)^{k+1-i}B_i  -\sum _{i=0}^{k+1}\binom{k+1}{i} (x+1)^{k+1-i}B_i =$$
$$(\ell^{k+1}-1)x^{k+1}+\frac{(k+1)}{2}(\ell^{k}-1)x^k+\frac{(k+1)k}{12}(\ell^{k-1}-1)x^{k-1}+\cdots.$$

Note that $H(0)=0$ and the multiplicity of $0$ as root of $H(x)$ is $1$ if $k+1$ is odd and $2$ if $k+1$ is even.

\begin{proposition}
\label{NumberOfRoots}
The polynomial $H(x)$ has at least three distinct roots.
\end{proposition}

\begin{proof}
Let $0$ be a root of multiplicity $r=1,2$ of $H(x)$ and suppose that $H(x)$ has only two distinct  roots. Then
$$\frac{H(x)}{\ell^{k+1}-1} =x^r(x+\alpha)^{k+1-r}$$
for some $\alpha$. This means that
$$\alpha(k+1-r)=\frac{(k+1)(\ell^{k}-1)}{2(\ell^{k+1}-1)}, \qquad \alpha^2\binom{k+1-r}{2}=\frac{(k+1)k(\ell^{k-1}-1)}{12(\ell^{k+1}-1)}.$$
This implies that 
$$ \left(\frac{(k+1)(\ell^{k}-1)}{2(\ell^{k+1}-1)(k+1-r)}\right)^2\cdot \binom{k+1-r}{2}=\frac{(k+1)k(\ell^{k-1}-1)}{12(\ell^{k+1}-1)}$$
and therefore
\begin{equation}\label{Eq1}
 r=k\left(1-\frac{2(\ell^{k-1}-1)(\ell^{k+1}-1)}{3(k+1)(\ell^k-1)^2-2k(\ell^{k-1}-1)(\ell^{k+1}-1)}\right).
\end{equation}
Since $\ell\geq 2$, we obtain $(\ell+\frac{1}{\ell}){\ell}^k\geq(2+\frac{1}{\ell})\ell^k>2\ell^k$. Using
$(\ell+\frac{1}{\ell}){\ell}^k>2\ell^k$, we get $\ell^{k+1}+\ell^{k-1}>2\ell^k$. From here, we get $-2\ell^k+1+\ell^{2k}>-\ell^{k+1}-\ell^{k-1}+\ell^{2k}+1$, whence
\begin{equation}\label{Eq2}
(\ell^k-1)^2>(\ell^{k-1}-1)(\ell^{k+1}-1).
\end{equation}
Now we consider the expression
\begin{equation*}
\frac{3(k+1)(\ell^k-1)^2-2k(\ell^{k-1}-1)(\ell^{k+1}-1)}{2(\ell^{k-1}-1)(\ell^{k+1}-1)},
\end{equation*}
which equals
\begin{equation}\label{Eq3}
\frac{3(k+1)(\ell^k-1)^2}{2(\ell^{k-1}-1)(\ell^{k+1}-1)}-k.
\end{equation}
Using \eqref{Eq2} and \eqref{Eq3}, we find that
\begin{equation*}
\frac{3(k+1)}{2} \cdot \frac{(\ell^k-1)^2}{(\ell^{k-1}-1)(\ell^{k+1}-1)}>\frac{3(k+1)}{2},
\end{equation*}
hence
\begin{equation*}
\frac{3(k+1)(\ell^k-1)^2}{2(\ell^{k-1}-1)(\ell^{k+1}-1)}-k>\frac{3(k+1)}{2}-k=\frac{(k+3)}{2}.
\end{equation*}
Thus we obtain
$$\frac{2(\ell^{k-1}-1)(\ell^{k+1}-1)}{3(k+1)(\ell^k-1)^2-2k(\ell^{k-1}-1)(\ell^{k+1}-1)}<\frac{2}{k+3}.$$

By \eqref{Eq1}, $r\in ]k-2,k[$, that is, the only possibility is $r=k-1$. 
If $r=1$, then $k=2$ and, by \eqref{Eq1},
$$1=2\left(1-\frac{2(\ell -1)(\ell^{3}-1)}{9(\ell^2-1)^2-4(\ell-1)(\ell^{3}-1)}\right),$$
which yields $8(\ell^2+\ell+1)=9(\ell+1)^2,$ a contradiction. 
If $r=2$, then $k=3$ and, by \eqref{Eq1},
$$2=3\left(1-\frac{2(\ell^2 -1)(\ell^{4}-1)}{12(\ell^3-1)^2-6(\ell^2-1)(\ell^{4}-1)}\right),$$
therefore  $(\ell+1)(\ell^2+1)=(\ell^2+\ell+1)^2$ which is a contradiction. 

Thus there are at least three distinct roots.
\end{proof}

In the following, we want to apply Theorem \ref{theo.2}. In particular, we want to establish sufficient conditions to avoid both patterns $(a)$ and $(b)$. 

\begin{theorem}[Main theorem]\label{maintheo}
Let $k,\ell$ be fixed integers such that $k\ge2$, $k\neq3$, $\ell\geq2$. Then all solutions of equation \eqref{xx} in integers $x,y,n$ with $x,y\geq1$, $n\geq2$ satisfy $\max\{x,y,n\}<C$ where $C$ is an effectively computable constant depending only on $\ell$ and $k$.
\end{theorem}

\section{Proof of Theorem \ref{maintheo}}
\begin{proof}
	
We distinguish the cases $k+1\, odd$ and $k+1\, even$.\\
 
\textbf{Case 1:} We suppose that $k+1$ is odd and then the multiplicity of the root $0$ is $r=1$. Then $t_0=\frac{n}{(n,1)}=n$. 

Also, using that $\sum\limits_{i=0}^{k-1}\binom {k} {i}B_{i}=0$ (see \cite[formulas (4.2) and (4.3)]{Rd}), the term of degree of $1$ of $H(x)$ is
$$(\ell-1)\sum_{i=0}^{k}\binom{k+1}{i}(k+1-i)B_i =(k+1)(\ell-1)\sum_{i=0}^{k}\binom{k}{i}B_i=(k+1)(\ell-1)B_k\neq 0,$$
where $B_i$ is $i$-th Bernoulli number.
\hfill

\begin{enumerate}
\item [(i)] Suppose $n\nmid k$.

Since $k$ is even and $n\nmid k$, the case $n=2$ is impossible. Therefore $n>2$ since $k$ is even and then there exists at least one root distinct from $0$ such that $n\nmid r_i$, where $r_i$ is its multiplicity. This yields $t_i=\frac{n}{(n,r_i)}\neq 1$ and therefore the bad patterns in Theorem \ref{theo.2} are avoided.

\hfill

\item [(ii)]  Suppose  $n\mid k$.  

If all the roots of the polynomial $H(x)$ have multiplicity $r_{i}$ divisible by $n$, then $H(x)/x=(\ell^{k+1}-1)f(x)^n$, where $f(x)=x^s+\sum_{i=0}^{s-1}\alpha_i x^i$, with $k=ns$. Since all coefficients of $H(x)/(x(\ell^{k+1}-1))$ are rational, $f(x)$ also must have rational coefficients. So the term $\alpha_0$ is rational and $\alpha_0^n=(k+1)(\ell-1)B_k/(\ell^{k+1}-1)$. According to the von Staudt-Clausen theorem, if $B_k\neq 0$ then $2$ divides the denominator but $4$ does not divide. In this case, if $2^a$ is the highest power that divides $\ell-1$, then $2^a$ is the highest power which also divides $\ell^{k+1}-1$. Therefore $2$ divides and $4$ does not divide the denominator of $\alpha_0^n$ which is a contradiction.

If there exists at least one root having  multiplicity $r_{i}$ not divisible by $n$, then the pattern does not correspond to $(n,1,1,1,1â\ldots)$. So this case is completed.
\end{enumerate}

\hfill

\textbf{Case 2:} Now suppose that $k+1$ is even and then the multiplicity of the root $0$ is $r=2$. Then $t_0=\frac{n}{(n,2)}\in \{n/2,n\}$. Also,
$B_{k-1}\neq 0$ and the term of degree $2$ in $H(x)$ is given by 
$$(\ell^2-1)\sum_{i=0}^{k-1}\binom{k+1}{i}\binom{k+1-i}{2}B_i =\binom{k+1}{2}(\ell^2-1)\sum_{i=0}^{k-1}\binom{k-1}{i}B_i=\binom{k+1}{2}(\ell^2-1)B_{k-1}\neq 0.$$

\hfill

\begin{enumerate}
\item[(i)] Suppose  $n\mid (k-1)$. 

If there exists at least one root having  multiplicity $r_{i}$ not divisible by $n$, then the pattern does not correspond to $(n,1,1,1,1,\ldots)$. 

If all the roots of the polynomial $H(x)$ have multiplicity $r_{i}$ divisible by $n$, then  $H(x)/(x^2(\ell^{k+1}-1))$ must be a monic polynomial which is also an $n$-power, then $H(x)/x^2=(\ell^{k+1}-1)f(x)^n$, where $f\in \mathbb{Q}[x]$. 
By the von Staudt-Clausen theorem again, a prime $p$ divides the denominator of $B_{k-1}$ if and only if $(p-1) \mid (k-1)$ and the denominator is square-free. Suppose that $2^e\mid \mid (k+1)/2$, that is $2^e\mid (k+1)/2$ and $2^{e+1}\nmid (k+1)/2$. Then
$$\frac{k+1}{2}\equiv 2^{e} \pmod {2^{e+1}}.$$

Now assume that $\ell$ is odd. Then 
$$\ell^2=1+8t \pmod {2^{e+1}},$$
therefore 
$$\frac{\ell^2-1}{\ell^{k+1}-1}=\frac{1}{\ell^{k-1}+\ell^{k-3}+\ell^{k-5}+\cdots+\ell^2+1}=\frac{1}{z},$$
where 
\begin{eqnarray*}
z&\equiv& 1+(1+8t)+(1+8t)^2+(1+8t)^3+\cdots+(1+8t)^{(k-1)/2}\\
&\equiv&  \frac{(1+8t)^{(k+1)/2}-1}{8t}\pmod {2^{e+1}} \equiv \frac{\frac{k+1}{2}8t+\frac{k^2-1}{8}(8t)^2+\cdots}{8t}\pmod {2^{e+1}}\\
&\equiv& \frac{k+1}{2}+\frac{k^2-1}{8}8t+\cdots \pmod {2^{e+1}}\equiv \frac{k+1}{2}\pmod {2^{e+1}}\equiv 2^{e} \pmod {2^{e+1}}.
\end{eqnarray*}
Thus $2^{e}\mid \mid z$ and then $2$ is the highest power of $2$ dividing the denominator of 
$$ \binom{k+1}{2}\frac{\ell^2-1}{\ell^{k+1}-1}B_{k-1}.$$ 
This is not possible since 
$$\alpha_0^n=\binom{k+1}{2}\frac{\ell^2-1}{\ell^{k+1}-1}B_{k-1}.$$

\hfill

Since the case when $\ell$ is even for the equation \eqref{xx} has already been considered in \cite{S}, the proof of case $(i)$ is completed.

\item[(ii)] Suppose  $n\nmid (k-1)$. Then $n$ must be at least $3$, since $k-1$ is even. 

\begin{itemize}
\item If $n=3$, then $t_0=\frac{n}{(n,2)}=3$ and there exists at least one root distinct from $0$ such that $n\nmid r_i$, where $r_i$ is its multiplicity. This yields $t_i=\frac{n}{(n,r_i)}\neq 1$ and therefore the bad patterns are avoided. 

\item If $n=4$, then $t_0=\frac{n}{(n,2)}=2$. Since $n\nmid k-1$, it can be still possible that there exists a unique root of multiplicity $r_i$, not divisible by $4$, but divisible by $2$, and all the other multiplicities are divisible by $4$. So we can write $H(x)/x^2=(\ell^{k+1}-1)f(x)^2$ where $f\in \mathbb{Q}[x]$, since $H(x)$ has, apart from $0$, one root of multiplicity $2$, and all the other multiplicities are divisible by $4$. 

Here we distinguish two cases. First we suppose that $\ell$ is odd. Then, following the steps in Case 2 $(i)$, we get that $2$ is the highest power of $2$ dividing the denominator of
$$ \binom{k+1}{2}\frac{\ell^2-1}{\ell^{k+1}-1}B_{k-1},$$ 
which contradicts with
$$\alpha_0^2=\binom{k+1}{2}\frac{\ell^2-1}{\ell^{k+1}-1}B_{k-1}.$$  
The case when $\ell$ is even has been considered in \cite{S}. So the proof of the case $(ii)$ with $n=4$ is completed.

\item If $n>4$, then $t_0=\frac{n}{(n,2)}>2$, and there exists at least one root distinct from $0$ such that $n\nmid r_i$, where $r_i$ is its multiplicity. This yields $t_i=\frac{n}{(n,r_i)}\neq 1$ and therefore the bad patterns are avoided. This finishes the proof of the theorem.
\end{itemize}
\end{enumerate}
\end{proof}
\section*{Acknowledgments}
We would like to thank to the referees for carefully reading our paper and
for giving such constructive comments which substantially helped improving
the quality of the paper. This work was started when the second author visited the Department of Mathematics and Informatics, University of Perugia, Perugia-Italy in an Erasmus Staff Exchange visit. He would like to thank to Professors Daniele Bartoli and Rita Vincenti for their kind hospitality. 
\bibliographystyle{abbrv}

\end{document}